\documentclass[11pt,leqno]{amsart}
\usepackage{epsfig}
\usepackage{amssymb}
\usepackage{amscd}
\usepackage[all,cmtip,matrix,arrow]{xy}
\usepackage{graphicx}

\newtheorem{theo}{Theorem}[section]
\newtheorem{lemma}[theo]{Lemma}

\newtheorem{prop}[theo]{Proposition}

\newtheorem{conj}[theo]{Conjecture}
\newtheorem{cor}[theo]{Corollary}
\newtheorem{remark}[theo]{Remark}

\numberwithin{equation}{section}

\mathchardef\mhyphen="2D

\def\A{{\mathbb A}}

\def\L{\mathbb{L}}

\def\C{\mathbb{C}}
\def\Z{\mathbb{Z}}
\def\Q{\mathbb{Q}}

\def\PP{{\mathbb P}}
\def\FF{{\mathbb F}}

\def\pre-tr{\operatorname{pre-tr}}

\def\Hom{\operatorname{Hom}}
\def\End{\operatorname{End}}

\newcommand{\xto}{\xrightarrow}

\newcommand{\onto}{\twoheadrightarrow}

\newcommand{\Coh}{\operatorname{Coh}}
\newcommand{\Proj}{\operatorname{Proj}}

\newcommand{\Km}{\operatorname{Km}}

\newcommand{\cO}{{\mathcal O}}

\newcommand{\cA}{{\mathcal A}}
\newcommand{\cB}{{\mathcal B}}

\newcommand{\cC}{{\mathcal C}}

\newcommand{\mh}{\mathfrak h}

\newcommand{\im}{\operatorname{Im}}

\newcommand{\codim}{\operatorname{codim}}

\newcommand{\rk}{\operatorname{rk}}

\newcommand{\Br}{\operatorname{Br}}

\newcommand{\add}{\operatorname{add}}

\newcommand{\Aut}{\operatorname{Aut}}

\newcommand{\ord}{\operatorname{ord}}

\newcommand{\tors}{\operatorname{tors}}

\newcommand{\pr}{\operatorname{pr}}

\newcommand{\Var}{\operatorname{Var}}
\newcommand{\HS}{\operatorname{HS}}
\newcommand{\Hdg}{\operatorname{Hdg}}

\usepackage{epsf}
\usepackage{amscd}

\newcommand{\one}{\mathbf{1}}

\title[Some remarks on L-equivalence of algebraic varieties]
{Some remarks on L-equivalence of algebraic varieties}

\author{Alexander I. Efimov}
\address{Steklov Mathematical Institute of RAS, Gubkin str. 8, Moscow 119991, Russia}
\email{efimov@mccme.ru}
\thanks{This work is supported by the RSF under a grant 14-50-00005.}

\begin{document}

\begin{abstract}In this short note we study the questions of (non-)L-equivalence of algebraic varieties, in particular, for abelian varieties and K3 surfaces. We disprove the original version of a conjecture of Huybrechts \cite[Conjecture 0.3]{H} stating that isogenous K3 surfaces are L-equivalent. Moreover, we give examples of derived equivalent twisted K3 surfaces, such that the underlying K3 surfaces are not L-equivalent. We also give examples showing that D-equivalent abelian varieties can be non-L-equivalent (the same examples were obtained independently in \cite{IMOU}). This disproves the original version of a conjecture of Kuznetsov and Schinder \cite[Conjecture 1.6]{KS}.

We deduce the statements on (non-)L-equivalence from the very general results on the Grothendieck group of an additive category, whose morphisms are finitely generated abelian groups. In particular, we show that in such a category each stable isomorphism class of objects contains only finitely many isomorphism classes. We also show that a stable isomorphism between two objects $X$ and $Y$ with $\End(X)=\Z$ implies that $X$ and $Y$ are isomorphic. 
\end{abstract}


\maketitle

\tableofcontents

\section{Introduction}

In this paper we consider the questions of L-equivalence of complex algebraic varieties, via integral Hodge realization. In particular, we disprove the original versions of conjectures of Huybrechts \cite[Conjecture 0.3]{H} and Kuznetsov-Schinder \cite[Conjecture 1.6]{KS}.

We recall the definition of the Grothendieck ring $K_0(\Var_{\C})$ of varieties. It is generated as an abelian group by the isomorphism classes $[X]$ of reduced separated schemes of finite type over $\C,$ subject to relation $[X]=[Z]+[U],$ whenever $Z\subset X$ is a closed subset, and $U=X\setminus Z.$ The product in $K_0(\Var_{\C})$ is given by the product of schemes. We denote by $\L$ the class of affine line $[\A^1].$

Algebraic varieties $X$ and $Y$ are called L-equivalent if the difference $[X]-[Y]$ vanishes in the localization $K_0(\Var_{\C})[\L^{-1}].$ In other words, for some $n>0$ we have $([X]-[Y])\cdot\L^n=0.$

The kernel of the homomorphism $K_0(\Var_{\C})\to K_0(\Var_{\C})[\L^{-1}]$ is hard to control, and finding non-trivial elements in it is an interesting and challenging task. The following conjectures have been made in the original versions of \cite{H} and \cite{KS} (both conjectures have been modified in the published versions of the papers).

\begin{conj}\label{conj:KS}\cite[original version of Conjecture 1.6]{KS} If $X$ and $Y$ are smooth projective simply connected varieties such that $D^b(X)\cong D^b(Y),$
then $X$ and $Y$ are L-equivalent.\end{conj} 

\begin{conj}\label{conj:H}\cite[original version of Conjecture 0.3]{H} Let $S$ and $S'$ be complex projective K3 surfaces. The following conditions are equivalent.

(i) $H^2(S,\Q)\cong H^2(S',\Q)$ (isomorphism of rational Hodge structures);

(ii) $\mh(S)\cong\mh(S')$ (isomorphism of rational Chow motives);

(iii) $[S]=[S']$ in (an appropriate localization of) $K_0(\Var_{\cC})[\L^{-1}].$\end{conj}

Below we disprove Conjecture \ref{conj:H} (implication (i)$\Rightarrow$(iii)) in its original version (Corollary \ref{cor:counterexamples_conj_H}). Moreover, our argument shows that Conjecture \ref{conj:H} does not hold even if we localize $K_0(\Var_{\C})[\L^{-1}]$ with respect to some set of monic polynomials in $\L$ (say, with respect to $(\L^n-1)$ for all $n>0,$ or, equivalently, with respect to $[GL_n(\C)],$ $n>0$), see Remark \ref{remark:further_loc}.

We give examples of derived equivalent twisted K3 surfaces such that the underlying K3 surfaces are non-L-equivalent (Corollary \ref{cor:twisted_but_not_L_existence}, Proposition \ref{prop:twisted_but_not_L}).

We also give counterexamples to Conjecture \ref{conj:KS} (Theorem \ref{th:L-equivalent_abelian}). The same examples were independently obtained in \cite{IMOU}.

Moreover, we show that both for abelian varieties and K3 surfaces there are only finitely many varieties in each L-equivalence class (Theorems \ref{th:fin_many_iso_classes_abelian_var}, \ref{th:fin_many_iso_classes_K3}).

Our tool to distinguish L-equivalence classes is the integral Hodge realization. Namely, we use the homomorphism from $K_0(\Var_{\C})[\L^{-1}]$ to the Grothendieck group $K_0^{add}(\HS_{\Z})$ of the additive category of pure integral Hodge structures, see Section \ref{sec:L-equivalence}. 

All the statements on non-L-equivalence are deduced from the general results on Grothendieck groups of additive categories whose morphisms are finitely generated abelian groups, see Section \ref{sec:K_0_of_additive}. In particular, we show that in such an additive category stable isomorphism between $X$ and $Y$ implies an isomorphism if $\End(X)=\Z$ (Theorem \ref{th:equals_classes_means_iso_additive}). We also show that in each stable isomorphism class there are only finitely many isomorphism classes (Theorem \ref{th:fin_many_iso_classes_additive}). Our tools are elementary: reduction to Krull-Schmidt categories (by tensoring with $\FF_p$) and Jordan-Zassenhaus theorem for orders in finite-dimensional semi-simple $\Q$-algebras.

{\noindent}{\bf Acknowledgements.} I am grateful to Alexander Kuznetsov, Valery Lunts and Dmitri Orlov for useful discussions. I am especially grateful to D. Huybrechts for his comments on the previous version of the paper and for pointing out a minor mistake in the proof of Corollary \ref{cor:counterexamples_conj_H}. I am also grateful to anonymous referee for useful comments and suggestions, in particular for suggesting Corollary \ref{cor:twisted_but_not_L_existence}.

\section{On the Grothendieck group of an additive category.}
\label{sec:K_0_of_additive}

 For an additive category $\cB$ and an associative unital ring $R,$ we denote by $\cB_R$ the tensor product $\cB\otimes_{\Z} R.$ This is an additive category with the same objects as $\cB,$ and the morphisms are given by $\cB_R(X,Y)=\cB(X,Y)\otimes_{\Z} R.$ The composition in $\cB_R$ is induced by the composition in $\cB$ and the product in $R.$

We recall that an additive category $\cB$ is called a Krull-Schmidt category if each object can be decomposed into a finite direct sum of objects having local endomorphism rings \cite{At, Au1, Au2, GR, Kr}. This in particular implies that any object $X\in\cB$ can be decomposed into a finite direct sum of indecomposable objects, and the collection of indecomposable summands is determined uniquely up to isomorphism. If $\cB$ is linear over a field, Karoubi complete, and has finite-dimensional morphism spaces, then $\cB$ is a Krull-Schmidt category.

For any additive category $\cB,$ we denote by $\bar{\cB}$ its Karoubi completion.
We will need the following standard facts.

\begin{prop}\label{prop:K_0_Krull_Schmidt}1) For a small additive Krull-Schmidt category $\cB,$ we have a natural isomorphism $K_0(\cB)\cong\Z^{(S)},$ where $S$ is the set of isomorphism classes of indecomposable objects in $\cB,$ and $\Z^{(S)}$ denotes the free abelian group generated by $S.$ In particular, for $X,Y\in\cB,$ the equality $[X]=[Y]$ in $K_0(\cB)$ implies $X\cong Y.$

2) For any small additive category $\cB,$ the natural morphism $K_0(\cB)\to K_0(\bar{\cB})$ is injective.\end{prop}

\begin{proof}Part 1) follows immediately from the unique decomposition into indecomposables.

Part 2) is proved as follows. Suppose that $X,Y\in\cB,$ and $[X]-[Y]\in K_0(\cB)$ is mapped to zero in $K_0(\bar{\cB}).$ Then there exists an object $Z\in \bar{\cB}$ such that $X\oplus Z\cong Y\oplus Z.$ Let us choose $Z'\in\bar{\cB},$ such that $Z\oplus Z'$is contained in the essential image of $\cB.$ Then we have $X\oplus(Z\oplus Z')\cong Y\oplus(Z\oplus Z'),$ hence $X$ and $Y$ have the same class in $K_0(\cB).$\end{proof}

From now on in this section we assume that $\cA$ is a small additive category, such that the abelian groups of morphisms $\cA(X,Y)$ are finitely generated. By the above discussion, for any prime number $p,$ the Karoubi completion $\overline{\cA_{\FF_p}}$ is a Krull-Schmidt category.

\begin{lemma}\label{lem:direct_summands_and_classes_K_0} Suppose that $X,Y\in \cA$ are two objects such that $[X]=[Y]$ in $K_0(\cA).$ Then $X$ is a retract of $Y^n$ for some $n>0.$\end{lemma}

\begin{proof}Let us denote by $I\subset\End(X)$ the  abelian subgroup generated by the compositions $X\to Y\to X.$ Clearly, $I$ is a two-sided ideal in $\End(X).$ It suffices to prove that $I=\End(X).$ Indeed, if this is the case, then for some $n>0$ we can find the morphisms $f_1,\dots,f_n:X\to Y,$ and $g_1,\dots,g_n:Y\to X$ such that $\sum\limits_{i=1}^n g_if_i=\one_X.$ Then the n-tuple $(f_1,\dots,f_n)$ (resp. $(g_1,\dots,g_n)$) defines a morphism $f:X\to Y^n$ (resp. $g:Y^n\to X$), such that $gf=\one_X.$

Now assume the contrary: the inclusion $I\subset\End(X)$ is strict. Then we can find a prime number $p$ such that $p$ acts non-trivially on the quotient $\End(X)/I.$ Let us denote by $\pr:\cA\to \cA_{\FF_p}$ the natural functor. The compositions $\pr(X)\to \pr(Y)\to \pr(X)$ generate the ideal $I'\subset \End(\pr(X)),$ which equals $\im(I\otimes\FF_p\to \End(X)\otimes \FF_p).$ By our assumptions on $p$ and $I,$ the inclusion $I'\subset\End(\pr(X))$ is strict. In particular, the objects $\pr(X)$ and $\pr(Y)$ are non-isomorphic. Since the Karoubi closure $\overline{\cA_{\FF_p}}$ is Krull-Schmidt, we deduce from Proposition \ref{prop:K_0_Krull_Schmidt} that $[\pr(X)]\ne [\pr(Y)]$ in $K_0(\cA_{\FF_p}),$ which contradicts the equality $[X]=[Y].$\end{proof}

For an arbitrary ring $R,$ we denote by $K_+(R)$ the monoid of isomorphism classes of finitely generated projective right R-modules.  If $R$ is noetherian, we denote by $K_0'(R)$ the Grothendieck group of the abelian category $\mod\mhyphen R$ of finitely generated right $R$-modules. 

\begin{theo}\label{th:equals_classes_means_iso_additive}Suppose that $X\in\cA$ is an object such that the map $K_+(\End(X))\to K_0'(\End(X))$ is injective (for example, this holds if $\End(X)\cong\Z,$ since $K_+(\Z)=\Z_{\geq 0}$ clearly injects into $K_0'(\Z)=\Z$). Then for any object $Y\in\cA$ such that $[X]=[Y]$ we have $X\cong Y.$\end{theo}

\begin{proof}By Proposition \ref{prop:K_0_Krull_Schmidt} 2), we may and will assume that $\cA$ is Karoubi complete. Let us denote by $\add(X)\subset \cA$ the full subcategory consisting of direct summands of direct sums of copies of $X.$ Also, we put $R:=\End(X),$ and denote by 
$\Proj_{f.g.}\mhyphen R$ 
the category of finitely generated projective right $R$-modules. 

Then the functor $\Hom(X,-):\cA\to\mod\mhyphen R$ restricts to an equivalence 
$\add(X)\xto{\sim}\Proj_{f.g.}\mhyphen R.$ Now take some $Y\in\cA$ such that $[X]=[Y].$ Applying the functor $\Hom(X,-)$ to the objects $X$ and $Y,$ we see that $[R]=[\Hom(X,Y)]$ in $K'_0(R)$ (in fact the classes of $R$ and $\Hom(X,Y)$ are already equal in the Grothendieck group of $\mod\mhyphen R$ as an additive category, hence also in $K_0'(R)$). Applying Lemma \ref{lem:direct_summands_and_classes_K_0} to $X$ and $Y$ (with roles inversed),  we obtain $Y\in\add(X),$ hence $\Hom(X,Y)\in \Proj_{f.g.}\mhyphen R.$ This in turn yields the isomorphism of $R$-modules $R\cong\Hom(X,Y),$ since the map $K_+(R)\to K'_0(R)$ is injective. Finally, since $Y\in\add(X),$ by applying a quasi-inverse equivalence $(\Hom(X,-))^{-1}:\Proj_{f.g.}\mhyphen R\xto{\sim}\add(X)$ we conclude that $X\cong Y$ in $\cA$.\end{proof}

In this paper we apply Theorem \ref{th:equals_classes_means_iso_additive} only to the cases when $\End(X)\cong\Z.$

We now turn to showing that each stable isomorphism class of objects in $\cA$ contains only finitely many isomorphism classes. For any unital ring $R,$ we denote by $R^{\times}$ the group of invertible elements in $R.$

\begin{lemma}\label{lem:surj_artinian}Let $\varphi:A\to B$ be a surjective homomorphism of artinian rings. Then the group homomorphism $\varphi^{\times}:A^{\times}\to B^{\times}$ is surjective.\end{lemma}

\begin{proof}First let us assume that $A$ is semi-simple. Then $A$ is isomorphic to the direct product of simple artinian rings (which are matrix algebras over skew-fields): $A\cong \prod\limits_{i=1}^k A_i.$ The ideal $\ker(\varphi)$ is isomorphic to a direct product of some of the $A_i$'s: $\ker(\varphi)=\prod\limits_{i\in M}A_i,$ $M\subseteq\{1,\dots,k\}.$ Hence $B\cong \prod\limits_{\substack{1\leq i\leq k\\ i\not\in M}}A_i$ is also semi-simple and we have a surjection $$A^{\times}\cong\prod\limits_{i=1}^k A_i^{\times}\onto\prod\limits_{\substack{1\leq i\leq k\\ i\not\in M}}A_i^{\times}\cong B^{\times}.$$

Now we consider the general case. Denote by $J_A\subset A,$ $J_B\subset B$ the Jacobson radicals. Since $J_A$ consists of nilpotent elements and $\varphi(J_A)\subset B$ is an ideal, we have $\varphi(J_A)\subseteq J_B.$ On the other hand, since $B/\varphi(J_A)$ is semi-simple (by the above discussion), we see that $\varphi(J_A)=J_B.$ Note that an element of $A$ is invertible if and only if its image in $A/J_A$ is invertible.

Now take any $x\in B^{\times}.$ We can choose some $y\in A^{\times}$ such that $\varphi(y)\in x+J_B.$ Finally, we can choose some $z\in J_A$ such that $\varphi(z)=x-\varphi(y).$ Then $y+z\in A^{\times}$ and $\varphi(y+z)=x.$ This proves the lemma.\end{proof}

\begin{lemma}\label{lem:surj_on_invertible}Let $f:R\to S$ be a surjective homomorphism of unital rings such that

(i) the ring $R$ is a finitely generated abelian group;

(ii) the ideal $I=\ker(f)$ is finite;

(iii) $f:R\to S$ is a split surjection in the category of abelian groups.

Then the induced group homomorphism $f^{\times}:R^{\times}\to S^{\times}$ is surjective.\end{lemma}

\begin{proof}We assume that $I\ne 0,$ otherwise there is nothing to prove. Let us put $N:=|I|,$ so that $NI=0.$ Choosing a splitting which exists by (iii), we get an isomorphism of abelian groups $R\cong I\oplus S.$ It induces an isomorphism $R/NR\cong I\oplus S/NS.$ It follows that we have an isomorphism of rings $R\to R/NR\times_{S/NS}S.$ This in turn implies an isomorphism of groups $R^{\times}\to (R/NR)^{\times}\times_{(S/NS)^{\times}}S^{\times}.$ By (i), the ring $R/NR$ is artinian. By Lemma \ref{lem:surj_artinian}, the group homomorphism $(R/NR)^{\times}\to (S/NS)^{\times}$ is surjective. Therefore, the homomorphism $f^{\times}:R^{\times}\to S^{\times}$ is also surjective.\end{proof}

For any ring $A,$ we say that a right $A$-module is $n$-generated if it can be generated by $n$ elements.

\begin{prop}\label{prop:fin_many_iso_classes_proj}Let $A$ be a unital ring which is a finitely generated abelian group. Then for any $n>0$ there are only finitely many isomorphism classes of $n$-generated right projective $A$-modules.\end{prop}

\begin{proof}Replacing $A$ by the matrix ring $M_n(A)$ (via Morita equivalence), we reduce to the case $n=1.$

First we consider the case when $A$ is torsion free over $\Z,$ so that $A\subset A_{\Q}$ is an order. Let us denote by $I\subset A_{\Q}$ the Jacobson radical, and put $J:=I\cap A.$ Then $A/J$ is an order in a semi-simple algebra $A_{\Q}/I,$ so by Jordan-Zassenhaus theorem \cite[Theorem 26.4]{R} there are only finitely many isomorphism classes of $(A/J)$-modules which are torsion free finitely generated over $\Z$ of rank at most $\dim (A_{\Q}/I).$ Thus, there are only 
finitely many isomorphism classes of projective cyclic $(A/J)$-modules. Finally, these are in bijection with isomorphism classes of projective cyclic $A$-modules, since $J$ is nilpotent.


Now let us consider the general case. Let us denote by $A_{tors}\subset A$ the ideal of $\Z$-torsion elements, and put $B:=A/A_{tors}.$ Recall that two idempotents $e,e'\in B$ are conjugate if and only if there are isomorphisms of right $B$-modules $eB\cong e'B,$ $(1-e)B\cong (1-e')B.$ Since $B$ is torsion free, we already know that there are only finitely many idempotents in $B$ up to conjugation. 

By Lemma \ref{lem:surj_on_invertible}, the homomorphism $A^{\times}\to B^{\times}$ is surjective. It follows that there are only finitely many idempotents in $A$ up to conjugation, since the fibers of the map of sets $A\to B$ are finite of cardinality $|A_{\tors}|.$ This proves the proposition.\end{proof}

\begin{theo}\label{th:fin_many_iso_classes_additive}For any object $X$ in $\cA,$ there are only finitely many isomorphism classes of objects $Y\in\cA$ such that $[X]=[Y]$ in $K_0(\cA).$\end{theo}

\begin{proof} By Lemma \ref{lem:direct_summands_and_classes_K_0} and its proof, for any $Y\in\cA$ with $[X]=[Y],$ $X$ is a retract of  $Y^n,$ where $n$ is the smallest number of generators of the right $\End_{\cA}(X)$-module $\Hom_{\cA}(X,Y).$ By Proposition \ref{prop:fin_many_iso_classes_proj}, the number of isomorphism classes of retracts of $Y^n$ is finite for a given $n.$ Thus, it suffices to find an upper bound for $n.$

Since $X$ and $Y$ become isomorphic in $\cA_{\FF_p}$ for all primes $p,$ we have isomorphisms
$\End_{\cA}(X)\otimes\FF_p\cong\Hom_{\cA}(X,Y)\otimes\FF_p.$ Now for any finitely generated abelian group $M,$ its smallest number of generators equals $\max_p(\dim_{\FF_p}(M\otimes_{\Z} \FF_p)).$ We conclude that the abelian groups $\End_{\cA}(X)$ and $\Hom_{\cA}(X,Y)$ have the same smallest numbers of generators, say $m.$ In particular, $m$ does not depend on $Y.$ The abelian group $\Hom_{\cA}(X,Y)$ is therefore $m$-generated as a right $\End_{\cA}(X)$-module. This proves the theorem.
\end{proof}

\section{L-equivalence of algebraic varieties}
\label{sec:L-equivalence}

By the famous result of Looijenga and Bittner \cite{Lo, Bi}, there is an alternative description of $K_0(\Var_{\C}):$ it is generated by isomorphism classes of smooth projective complex algebraic varieties, subject to relation $[X]+[Z]=[Y]+[E],$ whenever $Y$ is smooth projective, $Z\subset Y$ is a smooth subvariety, $X$ is a blow-up of $Y$ along $Z,$ and $E\subset X$ is the exceptional divisor.

For each integer $n\in\Z,$ we denote by $\HS_{\Z,n}$ the additive category of pure integral polarizable Hodge structures of weight $n.$ We also denote by $\HS_{\Z}:=\bigoplus\limits_{n\in\Z}\HS_{\Z,n}$ the additive category of graded pure polarizable Hodge structures. We denote by $K_0^{add}(\HS_{\Z})$ the Grothendieck group of the additive category $\HS_{\Z}.$
Of course, we have a direct sum decomposition $K_0^{add}(\HS_{\Z})=\bigoplus\limits_{n\in\Z}K_0^{add}(\HS_{\Z,n}).$

We claim that there is a natural homomorphism of abelian groups $\Hdg_{\Z}:K_0(\Var_{\C})\to K_0^{add}(\HS),$ given by $\Hdg_{\Z}([X])=[H^{\bullet}(X)]$ for a smooth projective $X.$ Indeed, within the above notation for $X,\,Y,\,Z$ and $E,$ we have natural isomorphisms of graded pure Hodge structures $$H^{\bullet}(E)\cong H^{\bullet}(Z)\oplus H^{\bullet}(Z)(-1)\oplus\dots\oplus H^{\bullet}(Z)(1-c),$$ and $$H^{\bullet}(X)\cong H^{\bullet}(Y)\oplus H^{\bullet}(Z)(-1)\oplus\dots\oplus H^{\bullet}(Z)(1-c),$$ where $c=\codim_Y Z$ (see \cite{Ma}). Hence, $\Hdg_{\Z}([X])+\Hdg_{\Z}([Z])=\Hdg_{\Z}([Y])+\Hdg_{\Z}([E]),$ which means that the homomorphism $\Hdg_{\Z}$ is well-defined.

Next, we observe that for a smooth projective $X$ we have $\Hdg_{\Z}([X]\cdot \L)=\Hdg_{\Z}([X])(-1).$ This follows immediately from the equality $[X]\cdot\L=[X\times\PP^1]-[X],$ and the isomorphism $H^{\bullet}(X\times\PP^1)\cong H^{\bullet}(X)\oplus H^{\bullet}(X)(-1).$ Since the Tate twist $-(-1)$ is an auto-equivalence of $\HS_{\Z},$ we can factor the homomorphism $\Hdg_{\Z}$ through $K_0(\Var_{\C})[\L^{-1}]$ via $\Hdg_{\Z}([X]\cdot \L^{-n})=[H^{\bullet}(X)(n)].$

We now use the resulting homomorphism \begin{equation}\label{eq:Hdg_Z_on_loc}\Hdg_{\Z}:K_0(\Var_{\C})[\L^{-1}]\to K_0^{add}(\HS_{\Z})\end{equation} to distinguish L-equivalence classes of varieties. The following result provides counterexamples to Conjecture \ref{conj:KS}. These counterexamples appeared independently in \cite[Corollary 7.4, Lemma 7.6]{IMOU}.

\begin{theo}\label{th:L-equivalent_abelian} Let $A$ be an abelian variety such that $\End(A)=\Z.$ Then 

1) for any abelian variety $A'$ such that $[H^1(A)]=[H^1(A')]$ in $K_0^{add}(\HS_{\Z,1})$ we have $A\cong A'.$

2) for any abelian variety $A',$ which is L-equivalent to $A,$ we have $A\cong A'.$ In particular, if $A$ is not principally polarizable, then $A$ is not L-equivalent to $\hat{A},$ but $D^b(A)\cong D^b(\hat{A}).$ In particular, D-equivalence does not imply L-equivalence.\end{theo}

\begin{proof} We observe that $\End_{\HS_{\Z,1}}(H^1(A))=\End(A)^{op}=\Z.$ Hence, part 1) follows directly from Theorem  \ref{th:equals_classes_means_iso_additive}.

Part 2) follows from part 1) by applying the composition 
$$\Hdg_{\Z}:K_0(\Var_{\C})[\L^{-1}]\to K_0^{add}(\HS_{\Z})\to K_0^{add}(\HS_{\Z,1}).$$
\end{proof}

\begin{theo}\label{th:fin_many_iso_classes_abelian_var}There are only finitely many abelian varieties in each L-equivalence class.\end{theo}

\begin{proof}This follows immediately from the homomorphism \eqref{eq:Hdg_Z_on_loc} and Theorem \ref{th:fin_many_iso_classes_additive}, since an abelian variety is reconstructed from the integral Hodge structure on $H^1.$\end{proof}

We now consider K3 surfaces.





For a K3 surface $X,$ we denote by $T(X)\subset H^2(X,\Z)$ the integral Hodge substructure of transcendental cycles. The sublattice $T(X)\oplus NS(X)\subset H^2(X,\Z)$ has finite index, say $D,$ which is also equal to the determinant of the restriction  of the intersection form to $T(X).$

\begin{prop}\label{prop:fin_many_for_T_X}Given a K3 surface $X,$ there are only finitely many isomorphism classes of K3 surfaces $X'$ such that the Hodge structures $H^2(X,\Z)$ and $H^2(X',\Z)$ are isomorphic (but not necessarily isometric).\end{prop}

\begin{proof}An isomorphism of Hodge structures $H^2(X,\Z)\xto{\sim}H^2(X',\Z)$ implies an isomorphism of sublattices $T(X)\oplus NS(X)\xto{\sim} T(X')\oplus NS(X'),$ and an isomorphism of Hodge structures $T(X)$ and $T(X').$ In particular, the intersection forms restricted to $T(X)$ and $T(X')$ have the same determinant $D.$

Let us put $R:=\End_{\HS_{\Z,2}}(T(X)),$ and $G:=R^{\times}=\Aut_{\HS_{\Z,2}}(T(X)).$
By \cite[Theorem 3.3]{Or} and \cite[Proposition 5.3]{BM}, there are only finitely many non-isomorphic K3 surfaces $X'$ with a Hodge isometry $T(X)\cong T(X').$  
Therefore, it suffices to show that there are only finitely many $G$-orbits in the set $S_D$ of all symmetric pairings $T(X)\otimes T(X)\to\Z(-2)$ with determinant $D$ (the pairing is required to be a morphism in $\HS_{\Z,4}$). 

By Zarhin's theorem \cite{Z}, the $\Q$-algebra $R_{\Q}$ is either a totally real number field $F,$ or an imaginary quadratic extension $E$ of a totally real number field, which we also denote by $F.$ In both cases the complex conjugation gives a well-defined involution $\sigma:R_{\Q}\to R_{\Q},$ which restricts to an involution on $R.$ We put $R':=R^{\sigma}=R\cap F.$ For any $a\in R_{\Q},$ and $x,y\in T(X),$ we have $\langle a(x),y\rangle=\langle x,\sigma(a)(y)\rangle,$ where $\langle\,,\,\rangle$ denotes the intersection form.

Now let $\beta$ denote any symmetric Hodge pairing on $T(X)$ of determinant $D.$ It defines an element $a(\beta)\in F$ such that $\beta(x,y)=\langle a(\beta)(x),y\rangle.$ Let us denote by $M_D$ the set of $a\in F$ such that both $Da$ and $Da^{-1}$ are in $R'.$ Clearly, we have $a(\beta)\in M_D.$   Let us note that $\beta_1$ and $\beta_2$ are in the same $G$-orbit if and only if there exists $b\in R^{\times}$ such that $a(\beta_1)=b\sigma(b)a(\beta_2).$ 

It suffices to show that there are only finitely many $(R^{'\times})^2$-orbits in $M_D,$ where the action is given by multiplication. Since $R'$ is an order in $F,$ we have $R'\subset \cO_F,$ and the group $R^{'\times}$ is a finite index subgroup of $\cO_F^{\times}.$ Since we are proving the assertion about $R'$ (and not about $X$), we may and will assume that $R'=\cO_F.$ By the Dirichlet Unit Theorem, the group $\cO_F^{\times}$ is finitely generated, hence its subgroup $(\cO_F^{\times})^2$ has finite index. It remains to note that the set $M_D$ has only finitely many $\cO_F^{\times}$-orbits. Indeed, for any $a\in M_D$ and any maximal ideal $\rho\subset\cO_F,$ we have $|\ord_{\rho}(a)|\leq \ord_{\rho}(D),$ hence $|M_D/\cO_F^{\times}|\leq \prod\limits_{\rho\ni D}(2\ord_{\rho}(D)+1).$ This proves the proposition.\end{proof}

\begin{theo}\label{th:fin_many_iso_classes_K3}For any K3 surface $X,$ there are only finitely many isomorphism classes of K3 surfaces $Y$ which are L-equivalent to $X.$\end{theo}

\begin{proof}By Theorem \ref{th:fin_many_iso_classes_additive}, there are only finitely many isomorphism classes of possible Hodge structures $H^2(Y,\Z)$ for K3 surfaces $Y$ that are L-equivalent to $X.$ By Proposition \ref{prop:fin_many_for_T_X}, there are only finitely many isomorphism classes of K3 surfaces with a given (isomorphism class of) integral Hodge structure on $H^2.$ This proves the theorem.\end{proof}

Now we disprove Conjecture \ref{conj:H}.

\begin{cor}\label{cor:counterexamples_conj_H}There are isogenous Kummer surfaces which are not L-equivalent.\end{cor}

\begin{proof}By \cite[Theorem 0.1]{HLOY}, for any Kummer surface $X$ there are only finitely many isomorphism classes of abelian surfaces $A$ such that $X\cong \Km A.$ Thus, there are countably many Kummer surfaces in each isogeny class. On the other hand, by Theorem \ref{th:fin_many_iso_classes_K3} there are only finitely many Kummer surfaces in each L-equivalence class. This proves the assertion.\end{proof}

Finally, we show that there are derived equivalent twisted K3 surfaces such that the underlying K3 surfaces are not L-equivalent. Recall that for a smooth complex projective variety $X$ its Brauer group $\Br(X)$ is the torsion part of $H^2_{an}(X,\cO_X^{\times}).$ For any element $\alpha\in\Br(X)$ we denote by $\Coh(X,\alpha)$ the abelian category of $\alpha$-twisted coherent sheaves \cite{C}, and by $D^b(X,\alpha)$ its derived category. By definition, a twisted K3 surface is a pair $(X,\alpha),$ where $X$ is a K3 surface, and $\alpha\in\Br(X)$ is an element. In particular, for any K3 surface $X$ we have a twisted K3 surface $(X,1).$

\begin{cor}\label{cor:twisted_but_not_L_existence}There exist derived equivalent twisted K3 surfaces such that the underlying K3 surfaces are not L-equivalent.\end{cor}

\begin{proof}By Corollary \ref{cor:counterexamples_conj_H}, there exist isogenous Kummer surfaces $S$ and $S'$ that are not L-equivalent. By \cite[Theorem 0.1]{H} there exist a finite sequence of K3 surfaces $S=S_0,S_1,\dots,S_n=S',$ and Brauer classes $\alpha\in \Br(S),$ $\alpha_1,\beta_1\in\Br(S_1),\dots,\alpha_{n-1},\beta_{n-1}\in\Br(S_{n-1}),$ $\alpha'\in\Br(S'),$ and derived equivalences $D^b(S_0,\alpha)\simeq D^b(S_1,\alpha_1),$ $D^b(S_1,\beta_1)\simeq D^b(S_2,\alpha_2),\dots,D^b(S_{n-1},\beta_{n-1})\simeq D^b(S_n,\alpha').$ Since $S$ and $S'$ are not L-equivalent, we have that for some $i\in\{0,1\dots,n-1\},$ $S_i$ and $S_{i+1}$ are not L-equivalent. This proves the assertion.\end{proof}

For completeness we also present more explicitly a large class of examples of non-L-equivalence for twisted derived equivalent K3 surfaces. Recall that for any K3 surface $X$ we have a natural isomorphism of abelian groups $\Br(X)\cong \Hom_{\Z}(T(X),\Q/\Z),$ see e.g. \cite[Section 4.1]{S}. In particular, the elements of order $n$ in $\Br(X)$ correspond to surjective homomorphisms $T(X)\to\Z/n\Z.$

\begin{lemma}\label{lem:trans_cycles_from_L-equiv}If the K3 surfaces $X$ and $X'$ are L-equivalent, then we have $[T(X)]=[T(X')]$ in $K_0^{add}(\HS_{\Z,2}).$ If moreover we have $\End_{\HS_{\Z,2}}(T(X))=\Z,$ then $T(X)\cong T(X')$ in $\HS_{\Z,2}.$\end{lemma}

\begin{proof}The first statement follows from the equality $[H^2(X,\Z)]=[H^2(X',\Z)],$ and from the existence of an additive endofunctor $T:\HS_{\Z}\to\HS_{\Z}$ which assigns to each polarizable pure integral Hodge structure its transcendental Hodge substructure. The second statement follows from Theorem \ref{th:equals_classes_means_iso_additive}.\end{proof}

\begin{prop}\label{prop:twisted_but_not_L}Let $X$ be a K3 surface such that $\rho(X)\geq 12$ and $\End_{\HS_{\Z,2}}(T(X))=\Z.$ Take any non-trivial Brauer class $\alpha\in \Br(X).$ Then there exists a K3 surface $Z$ such that $D^b(Z,1)\simeq D^b(X,\alpha),$ but $X$ and $Z$ are not L-equivalent.\end{prop}

\begin{proof}Existence of $Z$ with required derived equivalence follows from \cite[Proposition 7.3]{HS}. By the construction in loc. cit. and by \cite[Proposition 4.7]{H2}, we have an isomorphism of integral Hodge structures $T(Z)\cong T(X,\alpha),$ where $T(X,\alpha)=\ker(\alpha:T(X)\to\Z/n\Z),$ and $n>1$ is the order of $\alpha.$ By our assumption on the endomorphisms of $T(X)$ (and since $\rk_{\Z}T(X)\geq 2$), the integral Hodge structures $T(X)$ and $T(Z)$ are non-isomorphic. Hence, by Lemma \ref{lem:trans_cycles_from_L-equiv} $X$ and $Z$ are not L-equivalent.\end{proof}

\begin{remark}\label{remark:further_loc}We note that if one considers further localization of $K_0(\Var_{\C})[\L^{-1}]$ with respect to some set $W\subset \Z[\L]$ of monic polynomials, then equality $[X]=[Y]$ in $K_0(\Var_{\C})[\L^{-1}][W^{-1}]$ would also imply $\Hdg_{\Z}([X])=\Hdg_{\Z}([Y]).$ This is because there are no non-zero $\Hdg_{\Z}(W)$-torsion elements in $K_0^{add}(\HS_{\Z}).$ In particular, all our statements concerning (non-)L-equivalence are also valid in such further localizations.\end{remark}

\end{document}